\documentclass[a4paper,draft,reqno,12pt]{amsart}
\usepackage[english]{babel}
\usepackage{amsmath}
\usepackage{amssymb}
\usepackage{amscd}
\usepackage{amsthm}
\usepackage{euscript}
\newtheorem{proposition}{Proposition}
\newtheorem{lemma}{Lemma}
\newtheorem{theorem}{Theorem}

\theoremstyle{definition}
\newtheorem{example}{Example}
\theoremstyle{remark}
\newtheorem {remark}{Remark}

\DeclareMathOperator{\spec}{Spec}
\DeclareMathOperator{\homo}{Hom}
\DeclareMathOperator{\Aut}{Aut}
\DeclareMathOperator{\pol}{Pol}

\DeclareMathOperator{\Pic}{Pic}

\DeclareMathOperator{\supp}{Supp}

\DeclareMathOperator{\relint}{rel.int}
\DeclareMathOperator{\rank}{rank}

\def\GG{{\mathbb G}}

\def\CC{{\mathbb C}}
\def\KK{{\mathbb K}}
\def\TT{{\mathbb T}}
\def\ZZ{{\mathbb Z}}

\def\NN{{\mathbb N}}
\def\QQ{{\mathbb Q}}
\def\PP{{\mathbb P}}
\def\AA{{\mathbb A}}

\def\OO{\mathcal{O}}
\def\DD{\mathcal{D}}

\begin{document}
\date{}
\title[On rigidity of trinomial hypersurfaces]{On rigidity of factorial trinomial hypersurfaces}
\author{Ivan Arzhantsev}
\thanks{The research was supported by the grant RSF-DFG 16-41-01013}
\address{National Research University Higher School of Economics, Faculty of Computer Science, Kochnovskiy Proezd 3, Moscow, 125319 Russia}
\email{arjantsev@hse.ru}



\maketitle

\begin{abstract}
An affine algebraic variety $X$ is rigid if the algebra of regular functions $\KK[X]$ admits no nonzero locally nilpotent derivation. We prove that a factorial trinomial hypersurface is rigid if and only if every exponent in the trinomial is at least $2$.
\end{abstract}

\section{Introduction}
Let $\KK$ be an algebraically closed field of characteristic zero, $\GG_a$ be the additive group of the field $\KK$ and $X$ be an algebraic variety over $\KK$. The study of regular actions
$\GG_a\times X\to X$ is an actively developing area in the theory of algebraic transformation groups with applications to representation theory, commutative and differential algebra, complex analysis and dynamical systems. Recent results in this field can be found in~\cite{AFKKZ, AHHL, AL, FZ, F, FMJ, KPZ1, KPZ2, L} and many other works.

Assume that $X$ is an affine variety. If $A$ is a $\KK$-algebra, a derivation $D\colon A\to A$ is a linear map satisfying the Leibniz rule $D(ab)=D(a)b+aD(b)$ for all $a,b\in A$. A~derivation is called locally nilpotent if for each $a\in A$ there exists $m\in\ZZ_{>0}$ such that $D^m(a)=0$. It is well known that locally nilpotent derivations of the algebra $\KK[X]$ are in one-to-one correspondence with regular actions $\GG_a\times X\to X$, cf. \cite[Section~1.5]{F}. This correspondence allows to use geometric intuition in the theory of locally nilpotent derivations and provides a useful algebraic tool to deal with regular $\GG_a$-actions.

An affine variety $X$ is called rigid if it admits no non-trivial $\GG_a$-action. Equivalently,
$X$~is rigid if the algebra of regular functions $\KK[X]$ admits no nonzero locally nilpotent derivation.

One of the first results on rigidity of affine varieties was obtained by Kaliman and Zaidenberg. They proved \cite[Lemma~4]{KZ0} that the Pham-Brieskorn surface
$$
S_{k,l,m}=\{x_1^k+x_2^l+x_3^m=0\}\subseteq\CC^3, \qquad k,l,m \ge 2,
$$
is rigid if and only if it is not a dihedral surface $S_{2,2,m}$. The question whether the affine Fermat cubic threefold $x_1^3+x_2^3+x_3^3+x_4^3=0$ is rigid was posed~\cite{FZ} and stood open for 10~years. A positive answer is obtained recently in~\cite{CPW} by purely geometric methods. Many examples of rigid affine varieties are given in \cite[Chapter~9]{F}, \cite{CM}, \cite{FMJ}.

The aim of this note is to obtain a characterization of rigidity for some higher dimensional  hypersurfaces similar to the Pham-Brieskorn surfaces.
Fix positive integers $n_0,n_1,n_2$ and let $n=n_0+n_1+n_2$. For each $i=0,1,2$, fix a tuple $l_i\in\ZZ_{>0}^{n_i}$ and define a monomial
$$
T_i^{l_i}:=T_{i1}^{l_{i1}}\ldots T_{in_i}^{l_{in_i}}\in\KK[T_{ij}; \ i=0,1,2, \, j=1,\ldots,n_i].
$$
By a trinomial we mean a polynomial of the form $f=T_0^{l_0}+T_1^{l_1}+T_2^{l_2}$. A trinomial hypersurface $X$ is the zero set $\{f=0\}$ in the affine space $\KK^n$. One can check that $X$ is an irreducible normal affine variety of dimension $n-1$.

We say that an affine variety $X$ is factorial if the algebra $\KK[X]$ is a unique factorization domain. Returning to trinomials, for each $i=0,1,2$, let $d_i:=\text{gcd}(l_{i1},\ldots,l_{in_i})$.
Further we assume that $n_il_{i1}>1$ for $i=0,1,2$ or, equivalently, $f$ does not contain a linear term; otherwise the hypersurface $X$ is isomorphic to the affine space $\KK^{n-1}$. Under this assumption, the trinomial hypersurface $X$ is factorial if and only if any two of $d_0,d_1,d_2$ are coprime \cite[Theorem~1.1~(ii)]{HH}.

\begin{theorem} \label{thmain}
A factorial trinomial hypersurface $X$ is rigid if and only if every exponent in the trinomial
$T_0^{l_0}+T_1^{l_1}+T_2^{l_2}$ is at least 2.
\end{theorem}

Note that a complete characterization of rigid trinomial (not necessary factorial) hypersurfaces
in $\KK^4$ follows from~\cite[Section~5.4]{CM} and \cite[Theorem~9.1]{FMJ}.

Our motivation to study trinomials comes from toric geometry. Consider an effective action
$T\times X\to X$ of an algebraic torus $T$ on a variety $X$. Recall that the complexity of this action is $\dim X-\dim T$. Actions of complexity zero are actions with an open $T$-orbit. A normal variety with such an action is called toric.

If $X$ is a toric (not necessary affine) variety with the acting torus $T$, then the actions $\GG_a\times X\to X$ normalized by $T$ can be described combinatorially in terms of the so-called Demazure roots; see~\cite{De}, \cite[Section~3.4]{Oda} for the original approach and
\cite{L, AL, AHHL} for generalizations. It is easy to deduce from this description that every affine toric variety different from a torus admits a non-trivial $\GG_a$-action. Moreover, by~\cite[Theorem~2.1]{AKZ} every non-degenerate affine toric variety of dimension at least 2 is flexible in the sense of~\cite{AFKKZ}. In particular, it admits many $\GG_a$-actions.

The study of toric varieties is related to binomials, see e.g.~\cite[Chapter~4]{S}. At the same time, Cox rings establish a close relation between torus actions of complexity one and trinomials, see~\cite{HS, HHS, HH, AHHL, HW}. In particular, any trinomial hypersurface admits a torus action of complexity one. This observation plays a crucial role in the proof of Theorem~\ref{thmain}.

More precisely, a torus action of complexity one on $X$ induces an effective grading on $\KK[X]$
by a lattice of rank $n-2$. If $\KK[X]$ admits a nonzero locally nilpotent derivation, then it admits a homogeneous one. Our idea is to represent $X$ by a proper polyhedral divisor on a curve following Altmann and Hausen~\cite{AH} and to use a description of homogeneous locally nilpotent derivations due to Liendo~\cite{L}. This allows to prove the ``if" part of Theorem~\ref{thmain}, while the ``only if" part is elementary.

In the last section we reformulate our result in the case of a homogeneous factorial trinomial hypersurface in geometric terms (Theorem~\ref{t2}). Namely, it is shown by Kishimoto, Prokhorov and Zaidenberg~\cite{KPZ2} that rigidity of the affine cone over a projective variety $Q$ may be interpreted as absence of certain cylindrical open subsets on $Q$. Thus the projectivization $Q$ of a homogeneous factorial trinomial hypersurface $X$ contains no cylinder if and only if every exponent in the trinomial is at least 2.

Let us finish the Introduction with some concluding remarks. If a trinomial has the form
$$
T_0^{l_0}+T_1^{l_1}+T_{21}T_{22} \quad \text{or} \quad T_0^{l_0}+T_{11}^2+T_{21}^2,
$$
then the trinomial hypersurface $X$ is a so-called suspension over the affine space $\KK^{n-2}$, and by \cite{KZ} (see also~\cite[Theorem~3.2]{AKZ}) the variety $X$ is flexible and thus it admits many $\GG_a$-actions. It is an interesting problem to find all flexible trinomial hypersurfaces.

The fact that an affine variety $X$ is rigid allows to clarify the structure of the automorphism group $\Aut(X)$. In turn, the study of the group $\Aut(X)$ provides a description of automorphisms of a (complete) variety $Y$ having the algebra $\KK[X]$ as the Cox ring $R(Y)$. We plan to make it precise in a forthcoming publication.

\smallskip

The author is grateful to Charles Weibel, Mikhail Zaidenberg, and the referee for useful comments and suggestions.

\section{Preliminaries} \label{s1}

We begin with some elementary observations. The following lemma uses the idea of \cite[Section~4]{AHHL}.

\begin{lemma} \label{l=1}
If $l_{01}=1$, then the trinomial hypersurface $X$ is not rigid.
\end{lemma}

\begin{proof}
Let us define a derivation $\delta$ of the algebra $\KK[T_{ij}; \ i=0,1,2, \, j=1,\ldots,n_i]$ by
$$
\delta(T_{01})=l_{11}T_{11}^{l_{11}-1}T_{12}^{l_{12}}\ldots T_{1n_1}^{l_{1n_1}}, \quad
\delta(T_{11})=-T_{02}^{l_{02}}\ldots T_{0n_0}^{l_{0n_0}},
$$
and $\delta(T_{ij})=0$ for all other $i,j$. This derivation is locally nilpotent and
$\delta(T_0^{l_0}+T_1^{l_1}+T_2^{l_2})=0$, thus $\delta$ descents to the algebra
$$
\KK[X]=\KK[T_{ij}; \ i=0,1,2, \, j=1,\ldots,n_i]/(T_0^{l_0}+T_1^{l_1}+T_2^{l_2}).
$$
\end{proof}

Our next goal is to introduce a torus action of complexity one on a trinomial hypersurface.
Consider an integral $2\times n$ matrix
$$
L=
\begin{pmatrix}
-l_0 & l_1 & 0 \\
-l_0 & 0 & l_2
\end{pmatrix}.
$$

This matrix defines a homomorphism of lattices $L\colon \ZZ^n\to\ZZ^2$ with kernel $N$ isomorphic to $\ZZ^{n-2}$. Let us interpret $N$ as the lattice of one parameter subgroups of an algebraic torus $T$. The lattice of characters $M$ of the torus $T$ can be identified with the dual lattice
$\homo(N,\ZZ)$. We have two mutually dual exact sequences
$$
\begin{CD}
0 @>>> N @>>> \ZZ^n @>L>> \ZZ^2 @>>> 0
\end{CD}
$$
and
$$
\begin{CD}
0 @<<< M @<<< \ZZ^n @<L^t<< \ZZ^2 @<<< 0.
\end{CD}
$$

\smallskip

Letting $\deg(T_{ij})$ be the images in $M$ of the standard basis vectors $e_{ij}$ of the lattice $\ZZ^n$, we obtain an $M$-grading on the algebra $\KK[X]$. This grading gives rise to an effective action $T\times X\to X$ of complexity one on the trinomial hypersurface~$X$.

\begin{example} \label{Ex1}
Consider a trinomial hypersurface $T_{01}^2T_{02}^3+T_{11}^2+T_{21}^3=0$. We have
$$
L=
\begin{pmatrix}
-2 & -3 & 2 & 0 \\
-2 & -3 & 0 & 3
\end{pmatrix}
$$
and the kernel $N$ is generated by $(3,0,3,2)$ and $(0,2,3,2)$. So the action of a 2-torus on $X$ is given by the diagonal matrix
$$
\text{diag}(t_1^3, t_2^2, t_1^3t_2^3, t_1^2t_2^2).
$$
The $M$-grading on $\KK[X]$ is given by
$$
\deg(T_{01})=(3,0), \quad \deg(T_{02})=(0,2),\quad \deg(T_{11})=(3,3),\quad \deg(T_{21})=(2,2).
$$
\end{example}

We are going to study homogeneous locally nilpotent derivations on $\KK[X]$ with respect to this grading. Every derivation $\delta\colon\KK[X]\to\KK[X]$ admits a unique extension to a derivation of the quotient field $\KK(X)$. Denote by $\KK(X)^T$ the subfield of rational $T$-invariants.
Recall that a homogeneous derivation $\delta\colon\KK[X]\to\KK[X]$ is of vertical type if $\delta(\KK(X)^T)=0$ holds and of horizontal type otherwise.

In the next lemma we follow the proof of \cite[Theorem~4.3]{AHHL}.

\begin{lemma} \label{ver}
Let $X$ be a trinomial hypersurface. Then $\KK[X]$ admits no nonzero homogeneous locally nilpotent derivation of vertical type.
\end{lemma}

\begin{proof}
Let $\delta\colon\KK[X]\to\KK[X]$ be such a derivation. Since for any $i,j=0,1,2$ we have
$\frac{T_i^{l_i}}{T_j^{l_j}}\in\KK(X)^T$, we conclude that
$$
\delta\left(\frac{T_i^{l_i}}{T_j^{l_j}}\right)=0 \quad \text{and} \quad
\delta(T_i^{l_i})T_j^{l_j}=T_i^{l_i}\delta(T_j^{l_j}).
$$
Thus $T_i^{l_i}$ divides $\delta(T_i^{l_i})$ and by \cite[Corollary~1.20]{F} we have $\delta(T_i^{l_i})=0$. Then $\delta(T_{i1}^{l_{i1}}\ldots T_{in_i}^{l_{in_i}})=0$ and \cite[Principle~1(a)]{F} implies $\delta(T_{i1})=\ldots=\delta(T_{in_i})=0$. This proves that $\delta=0$.
\end{proof}

To deal with homogeneous locally nilpotent derivations of horizontal type, we need some preparations provided in the next section.

\section{Combinatorial description of factorial trinomial hypersurfaces}

The aim of this section is to represent a factorial trinomial hypersurface by a proper polyhedral divisor on a projective line $\PP^1$ following \cite{AH}. We begin with generalities
on proper polyhedral divisors, cf. \cite[Section~1.1]{AL}.

Let $N$ be a lattice  and $M=\homo(N,\ZZ)$ be its dual lattice. We let $M_{\QQ}=M\otimes\QQ$, $N_{\QQ}=N\otimes\QQ$, and we consider the natural duality pairing $M_{\QQ}\times N_{\QQ}\rightarrow \QQ$, $(m,p)\mapsto \langle m,p\rangle$.

Let $T=\spec\KK[M]$ be the algebraic torus associated to $M$ and $X=\spec\,A$ be an affine $T$-variety. The comorphism $A\rightarrow A\otimes \KK[M]$ induces an $M$-grading on
$A$ and, conversely, every $M$-grading on $A$ arises in this way.

Let $\sigma$ be a pointed polyhedral cone in $N_{\QQ}$. We define $\pol_{\sigma}(N_{\QQ})$ to be the set of all $\sigma$-polyhedra i.e., the set of all polyhedra in $N_{\QQ}$ that can be decomposed as the Minkowski sum of a bounded polyhedron and the cone $\sigma$.

Recall that $\sigma^\vee$ stands for the cone in $M_\QQ$ dual to $\sigma$. To a $\sigma$-polyhedron $\Delta\in\pol_{\sigma}(N_{\QQ})$ we associate its support function $h_{\Delta}\colon \sigma^\vee\rightarrow \QQ$ defined by
$$
h_{\Delta}(m)=\min\langle m,\Delta\rangle=\min_{p\in \Delta}\langle
m,p\rangle\,.
$$
If $\{v_i\}$ is the set of all vertices of $\Delta$, then the support function is given by
$$
h_{\Delta}(m)=\min_i\{\langle m,v_i\rangle\}\quad\mbox{for all}\quad m\in\sigma^\vee\,.
$$

A normal variety $Y$ is called semiprojective, if it is projective over an affine variety. A~$\sigma$-polyhedral divisor on $Y$ is a formal sum $\DD=\sum_{Z}\Delta_Z\cdot Z$, where $Z$ runs over all prime divisors on $Y$, $\Delta_Z\in\pol_{\sigma}(N_{\QQ})$, and $\Delta_Z=\sigma$ for all but finitely many $Z$. For $m\in\sigma^{\vee}$ we can evaluate $\DD$ in $m$ by letting $\DD(m)$
be the $\QQ$-divisor
$$
\DD(m)=\sum_{Z\subseteq Y} h_Z(m)\cdot Z\,,
$$
where $h_Z$ is the support function of $\Delta_Z$. A $\sigma$-polyhedral divisor $\DD$ is called proper if the following hold:
\begin{enumerate}
\item $\DD(m)$ is semiample and $\QQ$-Cartier for all $m\in\sigma^\vee$, and
\item $\DD(m)$ is big for all $m\in\relint(\sigma^\vee)$.
\end{enumerate}

Here $\relint(\sigma^\vee)$ denotes the relative interior of the cone $\sigma^\vee$. Furthermore, a $\QQ$-Cartier divisor $D$ on $Y$ is called semiample if there exists $r>0$ such that the linear system $|rD|$ is base point free, and big if there exists a divisor $D_0\in |rD|$, for some $r>0$, such that the complement $Y\setminus \supp D_0$ is affine.

Let $\chi^m$ denote the character of $T$ corresponding to the lattice vector $m$, and $\sigma^\vee_M$ denotes the semigroup $\sigma^\vee\cap M$. For a $\QQ$-divisor $D$ on $Y$, $\OO_Y(D)$ stands for the sheaf $\OO_Y(\lfloor D\rfloor)$.

\begin{theorem} \cite{AH} To any proper $\sigma$-polyhedral divisor $\DD$ on a semiprojective variety $Y$ one can associate a normal affine $T$-variety of dimension $\rank M+\dim Y$ given by $X[Y,\DD]=\spec A[Y,\DD]$, where
$$
A[Y,\DD]=\bigoplus_{m\in\sigma^{\vee}_M} A_m\chi^m,\quad
  \mbox{and}\quad A_m=H^0(Y,\OO_Y(\DD(m))\subseteq \KK(Y)\,.
$$

Conversely, any normal affine $T$-variety is isomorphic to $X[Y,\DD]$ for some semiprojective variety $Y$ and some proper $\sigma$-polyhedral divisor $\DD$ on $Y$.
\end{theorem}

We call $Y$ the base variety and the pair $(Y,\DD)$ the combinatorial data of $X$.

Let us recall a recipe from \cite[Section~11]{AH} how to determine a proper polyhedral divisor for a given normal affine $T$-variety $X$. Assume that $X$ is contained in $\KK^n$ as a closed subvariety and the action $T\times X\to X$ is given by a diagonal action of $T$ on $\KK^n$. Then $T$ is a subtorus of the torus $\TT$ of all invertible diagonal matrices. We assume that $X$ hits the open orbit $O$ of $\TT$ on $\KK^n$.

We begin with a description of a proper polyhedral divisor corresponding to $\KK^n$ as a $T$-variety. The inclusion $T\subseteq\TT$ corresponds to an inclusion $F\colon N\subseteq\NN$ of lattices of one parameter subgroups. We obtain a (non-canonical) split exact sequence
$$
\begin{CD}
0 @>>> N @>F>> \NN @>P>> N_W @>>> 0,
\end{CD}
$$
where $N_W:=\NN/N$, i.e., there is a homomorphism $S\colon \NN\to N$ such that $S\circ F=\text{id}$. Let $\Sigma_W$ be the coarsest fan in $(N_W)_{\QQ}$ refining all cones $P(\tau)$, where $\tau$ runs through all faces of the cone $\QQ_{\ge 0}^n$ in $\NN_\QQ$. Then the toric variety $W$ corresponding to $\Sigma_W$ is the base variety.

Given a one-dimensional cone $\rho\in\Sigma_W$, let $v_{\rho}\in\rho$ denote the first lattice vector, and define a polyhedron
$$
\Delta_{\rho}:=S(\QQ_{\ge 0}^n\cap P^{-1}(v_{\rho}))\subset N_\QQ.
$$
Then every $\Delta_{\rho}$ is a $\sigma$-polyhedron with $\sigma=\QQ_{\ge 0}^n\cap N_\QQ$.
Denote by
$R_W\subseteq \Sigma_W$ the set of one-dimensional cones and by $D_{\rho}\subseteq W$ the prime divisor corresponding to $\rho\in R_W$. Then the proper polyhedral divisor on $W$ representing $\KK^n$ as a $T$-variety is
$$
\DD_{\text{toric}}=\sum_{\rho\in R_W} \Delta_{\rho}\cdot D_{\rho}.
$$

Now we come to a presentation of the variety $X$. The base variety $Y$ is the normalization of the closure of the image of $X\cap\, O$ in $W$ under the projection of tori defined by $P\colon\NN\to N_W$. Further, the proper polyhedral divisor $\DD$ for the $T$-variety $X$ is obtained by pulling back $\DD_{\text{toric}}$ to $Y$.

Let us apply this construction to a factorial trinomial hypersurface $X$. The closed embedding $X\subseteq\KK^n$ is the desired one and the subtorus $T\subseteq\TT$ has codimension $2$.
As was mentioned in the Introduction, $X$ is factorial if and only if the integers $d_i:=\text{gcd}(l_{i1},\ldots,l_{in_i})$, $i=0,1,2$, are pairwise coprime. This condition means that the homomorphism $L\colon\ZZ^n\to\ZZ^2$ constructed in Section~\ref{s1} is surjective. Thus we may identify $\NN$ with $\ZZ^n$, $N_W$ with $\ZZ^2$ and the projection $P$ with the map $L$.

It follows that the fan $\Sigma_W$ is the standard fan of the projective plane $\PP^2$ i.e., it has as its maximal cones
$$
\text{cone}((1,0),(0,1)), \quad
\text{cone}((0,1),(-1,-1)), \quad
\text{cone}((-1,-1),(1,0)).
$$

Let us denote the coefficient of the proper polyhedral divisor $\DD_{\text{toric}}$ at the
divisors corresponding to the rays
$$
\rho_0=\QQ_{\ge 0} (-1,-1), \quad \rho_1=\QQ_{\ge 0} (1,0) \quad \text{and} \quad
\rho_2=\QQ_{\ge 0} (0,1)
$$
by $\Delta_0$, $\Delta_1$ and $\Delta_2$ respectively. They have the form
$S(\QQ_{\ge 0}^n\cap L^{-1}(v_i))\subset N_\QQ$, where
$$
v_0=(-1,-1), \quad v_1=(1,0), \quad v_2=(0,1).
$$

\begin{proposition} \label{ppdiv}
Let $X$ be a factorial trinomial hypersurface $T_0^{l_0}+T_1^{l_1}+T_2^{l_2}=0$ in $\KK^n$ with
torus action of complexity one as described in Section~\ref{s1}. Then the combinatorial data of $X$ are $(\PP^1,\DD)$, where
$$
\DD=\Delta_1\cdot\{0\}+\Delta_2\cdot\{1\}+\Delta_0\cdot\{\infty\}
$$
with $\Delta_i=S(\QQ_{\ge 0}^n\cap L^{-1}(v_i))$ and $v_0=(-1,-1), v_1=(1,0), v_2=(0,1)$.
\end{proposition}

\begin{proof}
On the tori $\TT\subseteq\KK^n$ and $(\KK^{\times})^2\subseteq\PP^2$ the projection is given by
$$
(t_0,t_1,t_2)\mapsto (t_0^{-l_0}t_1^{l_1}, t_0^{-l_0}t_2^{l_2}), \quad \text{where}
\quad t_i=(t_{i1},\ldots,t_{in_i}) \quad \text{and} \quad
t_i^{l_i}=t_{i1}^{l_{i1}}\ldots t_{in_i}^{l_{in_i}}.
$$
Hence the closure of the image of $X\cap O$ in the homogeneous coordinates $\left[w_0:w_1:w_2\right]$
on $\PP^2$ is given by $w_0+w_1+w_2=0$. Thus the base variety $Y$ is $\PP^1$ and the corresponding proper polyhedral divisor is
$$
\DD=\Delta_1\cdot\{0\}+\Delta_2\cdot\{1\}+\Delta_0\cdot\{\infty\}.
$$
\end{proof}

\begin{remark} \label{vert}
Let $e_1,\ldots,e_n$ be the standard basis of the lattice $\ZZ^n$. By construction, we have
$L(e_k)=l_{ij}v_i$, where $i=0, j=k$ for $k\le n_0$, $i=1, j=k-n_0$ for $n_0<k\le n_0+n_1$ and
$i=2, j=k-n_0-n_1$ otherwise. It is easy to check that the vertices of the polyhedra $\Delta_i$
are precisely the points $S(\frac{1}{l_{ij}}e_k)$, where $k=1,\ldots,n$ and $i,j$ are defined by $k$ as above.
\end{remark}

\section{Proof of Theorem~\ref{thmain}}

The ``only if" part follows from Lemma~\ref{l=1}. To prove the ``if" part, we have to show that if $l_{ij}\ge 2$ for all $i,j$, then $\KK[X]$ admits no nonzero locally nilpotent derivation.

Assume that a finitely generated domain $A$ is graded by a lattice $M$. If $A$ admits a nonzero locally nilpotent derivation, then it admits a homogeneous one \cite[Lemma~1.10]{L}. So in order to prove that $A$ is rigid it suffices to check that $A$ admits no nonzero homogeneous locally nilpotent derivation.

By Lemma~\ref{ver}, there is no nonzero homogeneous locally nilpotent derivation of vertical type.
A description of homogeneous locally nilpotent derivations of horizontal type for an affine $T$-variety $X$ of complexity one in terms of the combinatorial data $(Y,\DD)$ is given in \cite{L}, see also \cite[Section~1.4]{AL}. To represent such a derivation, as a first step one needs to fix a vertex $v_z$ for every coefficient $\Delta_z$ of the polyhedral divisor $\DD$ in such a way that all but two of these vertices are in $N$, see \cite[Definition~1.8(iii)]{AL}.

By Remark~\ref{vert}, the polyhedra $\Delta_0$, $\Delta_1$ and $\Delta_2$ in the combinatorial data $(\PP^1,\DD)$ corresponding to $X$ have all vertices of the form
$$
S\left(\frac{1}{l_{ij}}(0,\ldots,1,\ldots,0)\right).
$$
Note that the map $L\oplus S\colon\ZZ^n\to\ZZ^2\oplus\ZZ^{n-2}$ is an isomorphism of the lattices.    Since the induced map $L\colon \QQ^n\to\QQ^2$ sends $\frac{1}{l_{ij}}(0,\ldots,1,\ldots,0)$ to a lattice point, we conclude that all vertices of $\Delta_0$, $\Delta_1$ and $\Delta_2$ are not in the lattice $N$. Hence $\KK[X]$ admits no nonzero homogeneous locally nilpotent derivations of horizontal type.

This completes the proof of Theorem~\ref{thmain}.

\section{Geometric interpretation} \label{gi}

Assume that a trinomial $f=T_0^{l_0}+T_1^{l_1}+T_2^{l_2}$ is homogeneous i.e.,
$$
l_{01}+\ldots+l_{0n_0}=l_{11}+\ldots+l_{1n_1}=l_{21}+\ldots+l_{2n_2}.
$$
We consider a projective hypersurface $Q$ given by $\{f=0\}$ in $\PP^{n-1}$. In this section we present an interpretation of rigidity in terms of geometry of the hypersurface $Q$ due to Kishimoto, Prokhorov, and Zaidenberg.

By a cylinder in $Q$ we mean an open subset $U\subseteq Q$ isomorphic to $Z\times\AA^1$ for some affine variety $Z$.

\begin{theorem} \label{t2}
Consider a homogeneous factorial affine trinomial hypersurface $X$ given by $T_0^{l_0}+T_1^{l_1}+T_2^{l_2}=0$. Then the projective hypersurface $Q$ contains no cylinder if and only if every exponent in the trinomial is at least 2.
\end{theorem}

To obtain this result we need to recall some notions. Let $H$ be an ample divisor on a normal projective variety $Q$. An $H$-polar cylinder in $Q$ is an open subset
$$
U=Q\setminus\supp(D)
$$
defined by an effective $\QQ$-divisor $D$ on $X$ such that $[D]\in\QQ[H]$ in $\Pic_\QQ(Q)$ and
$U$ is isomorphic to $Z\times\AA^1$ for some affine variety $Z$ \cite{KPZ1}.

\smallskip

The following result follows from \cite[Corollary~3.2]{KPZ2}.

\begin{theorem} \label{kpz}
Let $Q$ be a normal closed subvariety in $\PP^{n-1}$, $H$ be a hyperplane section of $Q$ and $X$ be the affine cone over $Q$ in $\KK^n$. Then $X$ is rigid if and only if $Q$ contains no $H$-polar cylinder.
\end{theorem}

\begin{proof}[Proof of Theorem~\ref{t2}]
Since the variety $X$ is factorial, every divisor on $Q$ is linearly equivalent to a multiple of $H$, see \cite[Exercise~6.3(c)]{H}. Thus any cylinder in $Q$ is automatically $H$-polar and
Theorem~\ref{kpz} shows that rigidity of $X$ is equivalent to absence of cylinders in $Q$.
We conclude that Theorem~\ref{t2} follows from Theorem~\ref{thmain}.
\end{proof}

\begin{example}
The threefold
$$
w_{01}^2w_{02}^3+w_{11}^2w_{12}^3+w_{21}^5=0
$$
in $\PP^4$ contains no cylinder.
\end{example}


\end{document}